\documentclass{article}
\usepackage{amsmath}

\newtheorem{theorem}{Theorem}

\newtheorem{lemma}{Lemma}
\newtheorem{remark}{Remark}
\newtheorem{conjecture}{Conjecture}

\newtheorem{notation}{Notation}
\newtheorem{proposition}{Proposition}
\newenvironment{proof}[1][Proof]{\noindent\textbf{#1.} }{\ \rule{0.5em}{0.5em}}
\input{tcilatex}
\begin{document}

\title{A logarithmic mean and intersections of osculating hyperplanes in $%
R^{n}$}
\author{Alan Horwitz \\
708A Putnam Blvd.\\
Wallingford, PA 19086-6701 \\
alh4@psu.edu}
\date{4/8/14}
\maketitle

\begin{abstract}
We discuss a special case of the means defined in \cite{H}. Let $C$ be the
curve in $R^{n}$ with vector equation $\hat{\alpha}(t)=\langle t,t\log
t,...,t{\large (}\log t{\large )}^{n-1}\rangle $. Let $0<a_{1}<\cdots <a_{n}$
and let $O_{k}$ be the osculating hyperplane to $C$ at $a_{k}$. Then we show
that $O_{1},...,O_{n}$ have a unique point of intersection, $%
P=(i_{1},...,i_{n})\in R^{n}$, and in particular, $i_{1}$ equals the mean $%
M(a_{1},...,a_{n})=(n-1)!\dsum\limits_{j=1}^{n}\tfrac{a_{j}}{\prod\limits 
_{\substack{ i=1  \\ i\neq j}}^{n}(\ln a_{j}-\ln a_{i})}=$ the logarithmic
mean of Neuman.

\end{abstract}

\section{Introduction}

For $n\geq 3$, let $C$ be the curve in $R^{n}$ with vector equation $\hat{%
\alpha}(t)=\langle x_{1}(t),...,x_{n}(t)\rangle $, let $W_{j,n}(t)=W\left(
x_{1}^{\prime }(t),...,x_{j-1}^{\prime }(t),x_{j+1}^{\prime
}(t),...,x_{n}^{\prime }(t)\right) =$

the Wronskian of $x_{1}^{\prime }(t),...,x_{j-1}^{\prime
}(t),x_{j+1}^{\prime }(t),...,x_{n}^{\prime }(t),j=1,...,n$, let $\hat{x}$
be the vector $\langle x_{1},...,x_{n}\rangle $, and let $\hat{n}(t)$ be the
vector $\langle W_{1,n}(t),-W_{2,n}(t),...,(-1)^{n+1}W_{n,n}(t)\rangle $. In 
\cite{H} we defined the \textbf{osculating hyperplane}, $O_{a}$, to $C$ at $%
t=a$ to be the hyperplane in $R^{n}$ with equation

\begin{equation*}
\hat{x}\cdot \hat{n}(a)=\hat{\alpha}(a)\cdot \hat{n}(a)\text{, assuming that 
}\hat{n}(a)\neq \hat{0}\text{.}
\end{equation*}%
It is not hard to show(see \cite{H}) that $O_{a}$ has $n$th order contact
with $C$ at $t=a$. That is, if $C_{a}(t)={\large (}\hat{\alpha}(t)-\hat{%
\alpha}(a){\large )}\cdot \hat{n}(a)$, then $C_{a}^{(j)}(a)=0$ for $%
j=0,1,...,n-1$. This generalizes the osculating plane in $R^{3}$, which has $%
3$rd order contact with $C$ at $a$. For example, if $\hat{\alpha}(t)=\langle
t,t^{2},t^{3},t^{4}\rangle $, then $W_{1,4}(t)=W\left(
2t,3t^{2},4t^{3}\right) =48t^{3},W_{2,4}(t)=W\left( 1,3t^{2},4t^{3}\right)
=72t^{2},W_{3,4}(t)=W\left( 1,2t,4t^{3}\right) =48t,W_{3,4}(t)=W\left(
1,2t,3t^{2}\right) =12$, and $\hat{n}(t)=\langle
48t^{3},-72t^{2},48t,-12\rangle $. The equation of the osculating hyperplane
at $t=1$ is $\langle x_{1},x_{2},x_{3},x_{4}\rangle \cdot \hat{n}(1)=\hat{%
\alpha}(1)\cdot \hat{n}(1)$ or $4x_{1}-6x_{2}+4x_{3}-x_{4}=1$. $%
C_{1}(t)=\langle t-1,t^{2}-1,t^{3}-1,t^{4}-1\rangle \cdot \langle
48,-72,48,-12\rangle =-12\left( \allowbreak t^{4}-4t^{3}+6t^{2}-4t+1\right) $%
, and it then follows that $C_{1}(1)=C_{1}^{\prime }(1)=C_{1}^{\prime \prime
}(1)=C_{1}^{\prime \prime \prime }(1)=0$.

In \cite{H} the author proved the following general result about defining
means using intersections of osculating hyperplanes to curves in $R^{n}$.

\begin{theorem}
\label{Osc}Let $C$ be the curve in $R^{n}$ with vector equation $\hat{\alpha}%
(t)=\langle x_{1}(t),...,x_{n}(t)\rangle ,t\in I=[a,b]$, where each $%
x_{k}\in C^{n-1}(I)$ and is strictly monotone on $I$.

Let $W_{j,n}(t)=W{\large (}x_{1}^{\prime }(t),...,x_{j-1}^{\prime
}(t),x_{j+1}^{\prime }(t),...,x_{n}^{\prime }(t){\large )}$ be the Wronskian
of $x_{1}^{\prime }(t),...,x_{j-1}^{\prime }(t),x_{j+1}^{\prime
}(t),...,x_{n}^{\prime }(t)$. Assume that every subset of $%
\{W_{1,n},...,W_{n,n}\}$ is an extended complete Chebyshev system on $I$.
Let $a=a_{1}<\cdots <a_{n}=b$ be $n$ given points in $I$, and let $O_{k}$ be
the osculating hyperplane to $C$ at $a_{k}$. Then

(1) \ $O_{1},...,O_{n}$ have a unique point of intersection, $P$, in $R^{n}$%
, and

(2) If $P=(i_{1},...,i_{n})$, then $a_{1}<x_{k}^{-1}(i_{k})<a_{n}$ for $%
k=1,2,...,n$
\end{theorem}

By Theorem \ref{Osc}, one can define $n$ symmetric means in $a_{1},...,a_{n}$
as follows: $M_{k}\left( a_{1},...,a_{n}\right) =x_{k}^{-1}\left(
i_{k}\right) ,k=1,...,n$. In particular, we showed in \cite{H} that if $%
x_{k}(t)=t^{k}$, $k=1,...,n-2$, $x_{n-1}(t)=\log t$, and $x_{n}(t)=\dfrac{1}{%
t}$, then $M_{n}(a_{1},...,a_{n})=P(a_{1},...,a_{n})$, where $P$ is the
logarithmic mean in $n$ variables defined by Pittenger \cite{P}. At the end
of \cite{H} we stated that perhaps another interesting generalization of the
logarithmic mean to $n$ variables would be $M_{1}(a_{1},...,a_{n})$, where $%
x_{k}(t)=t{\large (}\log t{\large )}^{k-1}$, $k=1,...,n$. We never pursued
that, but the point of this paper is to prove that $M_{1}(a_{1},...,a_{n})$
equals the following logarithmic mean in $n$ variables defined by Neuman 
\cite{N}: $L_{N}(a_{1},...,a_{n})=(n-1)!\dsum\limits_{j=1}^{n}\tfrac{a_{j}}{%
\prod\limits_{\substack{ i=1 \\ i\neq j}}^{n}(\ln a_{j}-\ln a_{i})}$. That
is, we show that Neuman's logarithmic mean equals the $x$ coordinate of the
intersection of the osculating hyperplanes to the curve $\hat{\alpha}%
(t)=\langle t,t\log t,...,t{\large (}\log t{\large )}^{n-1}\rangle $. $L_{N}$
was also defined in a different way(and unknowingly) by Xiao and Zhang \cite%
{XZ}. Mustonen \cite{MUS} gives a good summary of these connections and
other generalizations. See also the paper by Merikoski \cite{M}. The methods
used in this paper are decidedly different than those in the papers just
cited. We now state our main result.

\begin{theorem}
\textbf{\label{main}}For $n\geq 3$, let $C$ be the curve in $R^{n}$ with
vector equation $\hat{\alpha}(t)=\langle t,t\log t,...,t{\large (}\log t%
{\large )}^{n-1}\rangle $. Let $0<a_{1}<\cdots <a_{n}$ and let $O_{k}$ be
the osculating hyperplane to $C$ at $a_{k}$. Then $O_{1},...,O_{n}$ have a
unique point of intersection, $P=(i_{1},...,i_{n})\in R^{n}$, and 
\begin{equation}
M_{1}(a_{1},...,a_{n})=(n-1)!\dsum\limits_{j=1}^{n}\tfrac{a_{j}}{\prod\limits
_{\substack{ i=1 \\ i\neq j}}^{n}(\ln a_{j}-\ln a_{i})}.  \label{1}
\end{equation}%
Note that for $n=2$, the $x$ coordinate of the point of intersection of the
tangent lines to the curve $\hat{\alpha}(t)=\langle t,t\log t\rangle $ is
the well known logarithmic mean $L(a,b)=\tfrac{b-a}{\ln b-\ln a}$ in two
variables. So in a certain sense the logarithmic mean $L_{N}(a_{1},...,a_{n})
$ above is a natural generalization of the logarithmic mean in two variables
since it involves intersections of osculating hyperplanes to a curve in $%
R^{n}$ whose first two components are $t$ and $t\log t$, and where the
remaining components follow the \textquotedblleft natural
pattern\textquotedblright\ of the first two components.
\end{theorem}

\begin{remark}
Using the curve from Theorem \ref{main} and Theorem \ref{Osc}, one also
obtains means $M_{k}(a_{1},...,a_{n})=x_{k}^{-1}(i_{k}),k=2,...,n$. Of
course those means involve the inverse of the function $y=t{\large (}\log t%
{\large )}^{k},k\geq 2$, which is not an elementary function.
\end{remark}

\section{Preliminary Material}

If $f_{1},..,f_{n}$ are $n$ given functions of $t$, then we let 
\begin{equation*}
W(f_{1},..,f_{n})(t)=\left\vert 
\begin{array}{llll}
f_{1}(t) & f_{2}(t) & \cdots & f_{n}(t) \\ 
f\,_{1}^{\prime }(t) & f\,_{2}^{\prime }(t) & \cdots & f\,_{n}^{\prime }(t)
\\ 
\vdots & \vdots & \cdots & \vdots \\ 
f_{1}^{(n-1)}(t) & f_{2}^{(n-1)}(t) & \cdots & f_{n}^{(n-1)}(t)%
\end{array}%
\right\vert
\end{equation*}
denote the Wronskian determinant. Throughout the rest of the paper, we
define 
\begin{eqnarray}
x_{k}(t) &=&t{\large (}\log t{\large )}^{k-1},k=1,...,n.  \label{2} \\
W_{k,n}(t) &=&W{\large (}x_{1}^{\prime }(t),...,x_{k-1}^{\prime
}(t),x_{k+1}^{\prime }(t),...,x_{n}^{\prime }(t){\large ),}k=1,...,n.  \notag
\end{eqnarray}

\begin{lemma}
\textbf{\label{L1}}For $r\geq 2$ 
\begin{equation}
x_{k+1}^{(r)}(t)=k\sum\limits_{j=1}^{r-1}\dfrac{(-1)^{j+1}(j-1)!\dbinom{r-2}{%
j-1}}{t^{j}}x_{k}^{(r-j)}(t).  \label{3}
\end{equation}
\end{lemma}

\begin{proof}
We use induction in $r$. So suppose that (\ref{3}) holds for some positive
integer $r\geq 2$.

$x_{k+1}^{(r+1)}(t)=\dfrac{d}{dt}x_{k+1}^{(r)}(t)=k\dfrac{d}{dt}%
\dsum\limits_{j=1}^{r-1}\tfrac{(-1)^{j+1}(j-1)!\tbinom{r-2}{j-1}}{t^{j}}%
x_{k}^{(r-j)}(t)=$

$k\dsum\limits_{j=1}^{r-1}(-1)^{j+1}(j-1)!\tbinom{r-2}{j-1}\dfrac{d}{dt}%
{\large (}t^{-j}x_{k}^{(r-j)}(t){\large )}=$

$k\dsum\limits_{j=1}^{r-1}(-1)^{j+1}(j-1)!\tbinom{r-2}{j-1}{\large (}%
t^{-j}x_{k}^{(r-j+1)}(t)-jt^{-j-1}x_{k}^{(r-j)}(t){\large )}=$

$k\dsum\limits_{j=1}^{r-1}\tfrac{(-1)^{j+1}(j-1)!\tbinom{r-2}{j-1}}{t^{j}}%
x_{k}^{(r+1-j)}(t)-k\dsum\limits_{j=1}^{r-1}\tfrac{(-1)^{j+1}j!\tbinom{r-2}{%
j-1}}{t^{j+1}}x_{k}^{(r-j)}(t)=$

$k\dsum\limits_{j=1}^{r-1}\tfrac{(-1)^{j+1}(j-1)!\tbinom{r-2}{j-1}}{t^{j}}%
x_{k}^{(r+1-j)}(t)-k\dsum\limits_{l=2}^{r}\tfrac{(-1)^{l}(l-1)!\tbinom{r-2}{%
l-2}}{t^{l}}x_{k}^{(r+1-l)}(t)=$

$k\dsum\limits_{j=1}^{r-1}\tfrac{(-1)^{j+1}(j-1)!\tbinom{r-2}{j-1}}{t^{j}}%
x_{k}^{(r+1-j)}(t)+k\dsum\limits_{j=2}^{r}\tfrac{(-1)^{j+1}(j-1)!\tbinom{r-2%
}{j-2}}{t^{j}}x_{k}^{(r+1-j)}(t)$.

Using the identity $\tbinom{r-2}{j-1}+\tbinom{r-2}{j-2}=\tbinom{r-1}{j-1}$,
we have 
\begin{equation*}
(j-1)!\tbinom{r-2}{j-1}+(j-1)!\tbinom{r-2}{j-2}=(j-1)!\tbinom{r-1}{j-1},
\end{equation*}%
which implies that

$k\dsum\limits_{j=1}^{r-1}\tfrac{(-1)^{j+1}(j-1)!\tbinom{r-2}{j-1}}{t^{j}}%
x_{k}^{(r+1-j)}(t)+k\dsum\limits_{j=2}^{r}\tfrac{(-1)^{j+1}(j-1)!\tbinom{r-2%
}{j-2}}{t^{j}}x_{k}^{(r+1-j)}(t)=$

$k\dsum\limits_{j=2}^{r-1}\tfrac{(-1)^{j+1}(j-1)!\tbinom{r-1}{j-1}}{t^{j}}%
x_{k}^{(r+1-j)}(t)+\dfrac{k}{t}x_{k}^{(r)}(t)+\dfrac{k}{t^{r}}%
(-1)^{r+1}(r-1)!x_{k}^{(1)}(t)=$

$k\dsum\limits_{j=1}^{r}\dfrac{(-1)^{j+1}(j-1)!\tbinom{r-1}{j-1}}{t^{j}}%
x_{k}^{(r+1-j)}(t)$.

To start the induction, $x_{k+1}^{\prime }(t)=\dfrac{k}{t}x_{k}(t)+\dfrac{1}{%
t}x_{k+1}(t)$, which implies that

$x_{k+1}^{\prime \prime }(t)=\dfrac{k}{t}x_{k}^{\prime }(t)-\dfrac{k}{t^{2}}%
x_{k}(t)+\dfrac{1}{t}x_{k+1}^{\prime }(t)-\dfrac{1}{t^{2}}x_{k+1}(t)=\dfrac{k%
}{t}x_{k}^{\prime }(t)-\dfrac{k}{t^{2}}x_{k}(t)+\dfrac{1}{t}\left( \dfrac{k}{%
t}x_{k}(t)+\dfrac{1}{t}x_{k+1}(t)\right) -\dfrac{1}{t^{2}}x_{k+1}(t)=$

$\dfrac{k}{t}x_{k}^{\prime }(t)-\dfrac{k}{t^{2}}x_{k}(t)+\dfrac{k}{t^{2}}%
x_{k}(t)+\dfrac{1}{t^{2}}x_{k+1}(t)-\dfrac{1}{t^{2}}x_{k+1}(t)=\dfrac{k}{t}%
x_{k}^{\prime }(t)$, which is (\ref{3}) with $r=2$.
\end{proof}

\begin{notation}
\qquad Let $a_{r,j}=(-1)^{j+1}(j-1)!\dbinom{r-2}{j-1}$. Then Lemma \ref{L1}
can be written
\end{notation}

\begin{equation}
x_{k+1}^{(r)}(t)=k\sum\limits_{j=1}^{r-1}\dfrac{a_{r,j}}{t^{j}}%
x_{k}^{(r-j)}(t).  \label{4}
\end{equation}

\begin{lemma}
\textbf{\label{L2}}For $k\geq 2$, $x_{k}^{(r)}(1)=0$ for any $r\leq k-2$,
and for $r\geq 2$, $x_{r}^{(r-1)}(1)=(r-1)!$
\end{lemma}

\begin{proof}
For the first part, we use induction in $k$. So assume that $%
x_{k}^{(l)}(1)=0 $ for $l\leq k-2$. By (\ref{4}), $x_{k+1}^{(r)}(1)=k\sum%
\limits_{j=1}^{r-1}a_{r,j}x_{k}^{(r-j)}(1)$. Suppose that $r\leq k-1$. Then $%
r-j\leq k-j-1\leq k-2$, which implies that $x_{k}^{(r-j)}(1)=0$; Thus $%
x_{k+1}^{(r)}(1)=0$ whenever $r\leq (k+1)-2$. To start the induction,
consider $x_{k}^{\prime }(t)={\large (}\log t{\large )}^{k-2}{\large (}%
k-1+\log t{\large )}\Rightarrow x_{k}^{\prime }(1)=0$ for $k\geq 3$. $%
x_{k+1}^{\prime \prime }(t)=\dfrac{k}{t}x_{k}^{\prime }(t)$ for $k\geq
1\Rightarrow x_{k}^{\prime \prime }(1)=0$ for $k\geq 4$.

For the second part, we use induction in $r$. So assume that $%
x_{r}^{(r-1)}(1)=(r-1)!$. By (\ref{4}), $x_{r+1}^{(r)}(1)=r\sum%
\limits_{j=1}^{r-1}a_{r,j}x_{r}^{(r-j)}(1)=r\left(
a_{r,1}x_{r}^{(r-1)}(1)+\sum\limits_{j=2}^{r-1}a_{r,j}x_{r}^{(r-j)}(1)%
\right) =r(r-1)!$ since $x_{r}^{(r-j)}(1)=0$ for $j\geq 2$ by the first part
of Lemma \ref{L2} just proven. Thus $x_{r+1}^{(r)}(1)=r!$. To start the
induction, $x_{2}^{\prime }(1)=\allowbreak 1=1!$
\end{proof}

\begin{lemma}
\label{L3}$\dsum\limits_{k=1}^{n}\dfrac{(-1)^{k-1}}{(n+1-k)!(k-1)!}=\dfrac{%
(-1)^{n+1}}{n!}$
\end{lemma}

\begin{proof}
Follows immediately from the binomial expansion of $1+x$ with $x=-1$ and we
omit the details.
\end{proof}

\begin{lemma}
\label{L4}$\dsum\limits_{k=1}^{n}\left[ (-1)^{k-1}\dfrac{1}{(k-1)!}\left(
\dsum\limits_{j=0}^{n-k}\dfrac{x^{n-1-j}}{(n-k-j)!}\right) \right] =1$ and

$\dsum\limits_{k=2}^{n}\left[ (-1)^{k-1}\dfrac{1}{(k-2)!}\dsum%
\limits_{j=0}^{n-k}\dfrac{x^{n-j-2}}{(n-k-j)!}\right] =-1$
\end{lemma}

\begin{proof}
Let $d_{n}=\dsum\limits_{k=1}^{n}\left[ (-1)^{k-1}\tfrac{1}{(k-1)!}%
\dsum\limits_{j=0}^{n-k}\tfrac{x^{n-1-j}}{(n-k-j)!}\right] $. Using
induction in $n$, we assume that $d_{n}=1$.

$d_{n+1}=\dsum\limits_{k=1}^{n+1}\left[ (-1)^{k-1}\tfrac{1}{(k-1)!}\left(
\dsum\limits_{j=0}^{n+1-k}\tfrac{x^{n-j}}{(n+1-k-j)!}\right) \right] =$

$(-1)^{n}\dfrac{1}{n!}x^{n}+\dsum\limits_{k=1}^{n}\left[ (-1)^{k-1}\tfrac{1}{%
(k-1)!}\left( \dsum\limits_{j=0}^{n+1-k}\tfrac{x^{n-j}}{(n+1-k-j)!}\right) %
\right] =$

$(-1)^{n}\dfrac{1}{n!}x^{n}+\dsum\limits_{k=1}^{n}\left[ (-1)^{k-1}\tfrac{1}{%
(k-1)!}\left( \tfrac{x^{n}}{(n+1-k)!}+\dsum\limits_{j=1}^{n+1-k}\tfrac{%
x^{n-j}}{(n+1-k-j)!}\right) \right] =$

$(-1)^{n}\dfrac{1}{n!}x^{n}+x^{n}\dsum\limits_{k=1}^{n}\tfrac{(-1)^{k-1}}{%
(n+1-k)!(k-1)!}+\dsum\limits_{k=1}^{n}\left[ (-1)^{k-1}\tfrac{1}{(k-1)!}%
\dsum\limits_{j=0}^{n-k}\tfrac{x^{n-1-j}}{(n-k-j)!}\right] =$

$(-1)^{n}\dfrac{1}{n!}x^{n}+x^{n}\dsum\limits_{k=1}^{n}\tfrac{(-1)^{k-1}}{%
(n+1-k)!(k-1)!}+1=(-1)^{n}\dfrac{1}{n!}x^{n}+x^{n}\dfrac{(-1)^{n+1}}{n!}+1$%
(by Lemma \ref{L3}) $=\allowbreak 1$.

Since $d_{1}=1$, that completes the proof of the first part of Lemma \ref{L4}%
. The proof of the second part of Lemma \ref{L4} is similar and we omit it.
The second part of Lemma \ref{L4} also follows from the first part after
some manipulations and Lemma \ref{L3}.
\end{proof}

Before proving one of our main results, we introduce the functions $v_{k,n}$
below.

\begin{lemma}
\label{L6}For $1\leq k\leq n$ and $n\geq $ $3$, let 
\begin{equation*}
v_{k,n}(t)=\dfrac{\dprod\limits_{r=0}^{n-1}r!}{(k-1)!}\dfrac{1}{%
t^{(n-2)(n-1)/2}}\dsum\limits_{j=0}^{n-k}\dfrac{\left( \ln t\right) ^{n-k-j}%
}{(n-k-j)!},t>0.
\end{equation*}

Then 
\begin{eqnarray*}
v_{k+1,n+1}(t) &=&\dfrac{n!}{k}\dfrac{1}{t^{n-1}}v_{k,n}(t) \\
v_{1,n+1}(t) &=&\left( \dprod\limits_{r=0}^{n-1}r!\right) \dfrac{\left( \ln
t\right) ^{n}}{t^{n(n-1)/2}}+\dfrac{n!}{t^{n-1}}v_{1,n}(t).
\end{eqnarray*}
\end{lemma}

\begin{proof}
$v_{k+1,n+1}(t)=\tfrac{\tprod\limits_{r=0}^{n-1}r!}{k!}\tfrac{\left( \ln
t\right) ^{n-k-j}}{t^{n(n-1)/2}}\dsum\limits_{j=0}^{n-k}\tfrac{1}{(n-k-j)!}=$

$\tfrac{\tprod\limits_{r=0}^{n-1}r!}{k!}\dfrac{\left( \ln t\right) ^{n-k-j}}{%
t^{(n-1)(n-2)/2}t^{n-1}}\dsum\limits_{j=0}^{n-k}\tfrac{1}{(n-k-j)!}=$

$\tfrac{n!}{k}\tfrac{1}{t^{n-1}}\tfrac{\tprod\limits_{r=0}^{n-1}r!}{(k-1)!}%
\tfrac{1}{t^{(n-2)(n-1)/2}}\dsum\limits_{j=0}^{n-k}\tfrac{\left( \ln
t\right) ^{n-k-j}}{(n-k-j)!}=\tfrac{n!}{k}\tfrac{1}{t^{n-1}}v_{k,n}(t)$.

$v_{1,n}(t)=\left( \tprod\limits_{r=0}^{n-1}r!\right) \tfrac{1}{%
t^{(n-2)(n-1)/2}}\dsum\limits_{j=0}^{n-1}\tfrac{\left( \ln t\right) ^{n-1-j}%
}{(n-1-j)!}\Rightarrow v_{1,n+1}(t)=$

$\left( \tprod\limits_{r=0}^{n}r!\right) \tfrac{1}{t^{n(n-1)/2}}%
\dsum\limits_{j=0}^{n}\tfrac{\left( \ln t\right) ^{n-j}}{(n-j)!}=$

$\left( \tprod\limits_{r=0}^{n}r!\right) \dfrac{1}{t^{n(n-1)/2}}\left( 
\tfrac{\left( \ln t\right) ^{n}}{n!}+\dsum\limits_{j=1}^{n}\dfrac{\left( \ln
t\right) ^{n-j}}{(n-j)!}\right) =$

$\left( \tprod\limits_{r=0}^{n}r!\right) \tfrac{\left( \ln t\right) ^{n}}{%
t^{n(n-1)/2}}\left( \tfrac{1}{n!}+\dsum\limits_{j=0}^{n-1}\tfrac{\left( \ln
t\right) ^{n-1-j}}{(n-1-j)!}\right) =$

$\left( \tprod\limits_{r=0}^{n-1}r!\right) \dfrac{\left( \ln t\right) ^{n}}{%
t^{n(n-1)/2}}+\left( \tprod\limits_{r=0}^{n}r!\right) \dfrac{1}{t^{n(n-1)/2}}%
\dsum\limits_{j=0}^{n-1}\tfrac{\left( \ln t\right) ^{n-1-j}}{(n-1-j)!}=$

$\left( \tprod\limits_{r=0}^{n-1}r!\right) \tfrac{\left( \ln t\right) ^{n}}{%
t^{n(n-1)/2}}+\left( \tprod\limits_{r=0}^{n-1}r!\right) \dfrac{n!}{t^{n-1}}%
\dfrac{1}{t^{(n-2)(n-1)/2}}\dsum\limits_{j=0}^{n-1}\tfrac{\left( \ln
t\right) ^{n-1-j}}{(n-1-j)!}=$

$\left( \tprod\limits_{r=0}^{n-1}r!\right) \tfrac{\left( \ln t\right) ^{n}}{%
t^{n(n-1)/2}}+\tfrac{n!}{t^{n-1}}v_{1,n}(t)$.
\end{proof}

\begin{proposition}
\label{P1}Let $v_{k,n}(t)$ be the functions from Lemma \ref{L6}, and define
the vector functions

\begin{eqnarray*}
\hat{\alpha}(t) &=&\langle x_{1}(t),...,x_{n}(t)\rangle \\
\hat{v}_{n}(t) &=&\langle
v_{1,n}(t),-v_{2,n}(t),...,(-1)^{n-1}v_{n,n}(t)\rangle .
\end{eqnarray*}%
Then 
\begin{eqnarray*}
\hat{\alpha}(t)\cdot \hat{v}_{n}(t) &=&\dfrac{\dprod\limits_{r=0}^{n-1}r!}{%
t^{n(n-3)/2}} \\
\hat{\alpha}^{(j)}(t)\cdot \hat{v}_{n}(t) &=&0\ \text{for }j=1,...,n-1.
\end{eqnarray*}
\end{proposition}

\begin{remark}
$\hat{\alpha}$ depends on $n$ as does $\hat{v}_{n}$, but we suppress this
dependence in our notation for convenience.
\end{remark}

\begin{proof}
Case 1: $j=0$

$\vec{\alpha}(t)\cdot \hat{v}_{n}(t)=\sum%
\limits_{k=1}^{n}(-1)^{k-1}x_{k}(t)v_{k,n}(t)=\sum%
\limits_{k=1}^{n}(-1)^{k-1}t{\large (}\ln t{\large )}^{k-1}v_{k,n}(t)=$

$\left( \tprod\limits_{r=0}^{n-1}r!\right) \dfrac{t}{t^{(n-2)(n-1)/2}}%
\dsum\limits_{k=1}^{n}\left[ (-1)^{k-1}\tfrac{1}{(k-1)!}\left(
\dsum\limits_{j=0}^{n-k}\tfrac{\left( \ln t\right) ^{n-1-j}}{(n-k-j)!}%
\right) \right] =\dfrac{\tprod\limits_{r=0}^{n-1}r!}{t^{n(n-3)/2}}$ by Lemma %
\ref{L4} with $x=\ln t$.

Case 2: $j=1$

$x_{k}(t)=t{\large (}\log t{\large )}^{k-1}\Rightarrow x_{k}^{\prime
}(t)=t(k-1){\large (}\ln t{\large )}^{k-2}\dfrac{1}{t}+{\large (}\ln t%
{\large )}^{k-1}\Rightarrow $

$x_{k}^{\prime }(t)={\large (}\ln t{\large )}^{k-2}{\large (}k-1+\ln t%
{\large )}$.

Note that if $k=1$, $x_{k}^{\prime }(t)=1$ for all $t$, including $t=1$.

$\hat{\alpha}^{\prime }(t)\cdot \hat{v}(t)=\dsum%
\limits_{k=1}^{n}(-1)^{k-1}x_{k}^{\prime
}(t)v_{k,n}(t)=\dsum\limits_{k=1}^{n}(-1)^{k-1}{\large (}\ln t{\large )}%
^{k-2}{\large (}k-1+\ln t{\large )}v_{k,n}(t)=$

$\dfrac{\tprod\limits_{r=0}^{n-1}r!}{t^{(n-2)(n-1)/2}}\dsum\limits_{k=1}^{n}%
\left[ (-1)^{k-1}\dfrac{1}{(k-1)!}{\large (}k-1+\ln t{\large )}\left(
\dsum\limits_{j=0}^{n-k}\dfrac{\left( \ln t\right) ^{n-j-2}}{(n-k-j)!}%
\right) \right] =$

\begin{gather*}
\dfrac{\tprod\limits_{r=0}^{n-1}r!}{t^{(n-2)(n-1)/2}}\times {\large (}%
\dsum\limits_{k=2}^{n}\left[ (-1)^{k-1}\tfrac{1}{(k-2)!}\dsum%
\limits_{j=0}^{n-k}\tfrac{\left( \ln t\right) ^{n-j-2}}{(n-k-j)!}\right] + \\
\dsum\limits_{k=1}^{n}\left[ (-1)^{k-1}\tfrac{1}{(k-1)!}\dsum%
\limits_{j=0}^{n-k}\tfrac{\left( \ln t\right) ^{n-j-1}}{(n-k-j)!}\right] 
{\large )}=0
\end{gather*}

by Lemma \ref{L4} with $x=\ln t$.

Case 3: $2\leq j\leq n$(note that $j$ is fixed here)

We use induction in $n$. So assume that $\hat{\alpha}^{(l)}(t)\cdot \hat{v}%
_{n}(t)=\sum\limits_{k=1}^{n}(-1)^{k-1}x_{k}^{(l)}(t)v_{k,n}(t)=0$ for all $%
l=1,...,n-1$. We have to show that $\hat{\alpha}^{(j)}(t)\cdot \hat{v}%
_{n+1}(t)=0$.

Now $\hat{\alpha}^{(j)}(t)\cdot \hat{v}_{n+1}(t)=$ $\sum%
\limits_{k=1}^{n+1}(-1)^{k-1}x_{k}^{(j)}(t)v_{k,n+1}(t)=\sum%
\limits_{k=0}^{n}(-1)^{k}x_{k+1}^{(j)}(t)v_{k+1,n+1}(t)=x_{1}^{(j)}(t)v_{1,n+1}(t)+\sum\limits_{k=1}^{n}(-1)^{k}x_{k+1}^{(j)}(t)v_{k+1,n+1}(t) 
$. Since $x_{1}^{(j)}(t)=0$ for $j\geq 2$, we have 
\begin{gather}
\hat{\alpha}^{(j)}(t)\cdot \hat{v}_{n+1}(t)=  \label{8} \\
\sum\limits_{k=1}^{n}(-1)^{k}x_{k+1}^{(j)}(t)v_{k+1,n+1}(t).  \notag
\end{gather}

By Lemma \ref{L6}, for $j\geq 2$, $x_{k+1}^{(j)}(t)v_{k+1,n+1}(t)=k\left(
\sum\limits_{i=1}^{j-1}\tfrac{a_{i,j}}{t^{j}}x_{k}^{(j-i)}(t)\right) \dfrac{%
n!}{k}\dfrac{1}{t^{n-1}}v_{k,n}(t)=\dfrac{n!}{t^{n-1}}\left(
\sum\limits_{i=1}^{j-1}\tfrac{a_{i,j}}{t^{j}}x_{k}^{(j-i)}(t)\right)
v_{k,n}(t)$.

Then $\sum\limits_{k=1}^{n}(-1)^{k}x_{k+1}^{(j)}(t)v_{k+1,n+1}(t)=\dfrac{n!}{%
t^{n-1}}\sum\limits_{k=1}^{n}(-1)^{k}\left( \sum\limits_{i=1}^{j-1}\tfrac{%
a_{i,j}}{t^{j}}x_{k}^{(j-i)}(t)\right) v_{k,n}(t)=$

$\dfrac{n!}{t^{n-1}}\sum\limits_{k=1}^{n}(-1)^{k}\left(
\sum\limits_{i=1}^{j-1}\tfrac{a_{i,j}}{t^{j}}x_{k}^{(j-i)}(t)v_{k,n}(t)%
\right) =$

$\dfrac{n!}{t^{n-1}}\sum\limits_{i=1}^{j-1}\left[ \tfrac{a_{i,j}}{t^{j}}%
\left( \sum\limits_{k=1}^{n}(-1)^{k}x_{k}^{(j-i)}(t)v_{k,n}(t)\right) \right]
=$

$\dfrac{n!}{t^{n-1}}\sum\limits_{i=1}^{j-1}\tfrac{a_{i,j}}{t^{j}}\left(
\sum\limits_{k=1}^{n}(-1)^{k}x_{k}^{(j-i)}(t)v_{k,n}(t)\right) =$

$-\dfrac{n!}{t^{n-1}}\sum\limits_{i=1}^{j-1}\tfrac{a_{i,j}}{t^{j}}\left(
\sum\limits_{k=1}^{n}(-1)^{k-1}x_{k}^{(j-i)}(t)v_{k,n}(t)\right) =$

$-\dfrac{n!}{t^{n-1}}\sum\limits_{i=1}^{j-1}\tfrac{a_{i,j}}{t^{j}}\left(
\sum\limits_{k=1}^{n}(-1)^{k-1}x_{k}^{(j-i)}(t)v_{k,n}(t)\right) $.

Since $j-i\leq n-1$ for $i\geq 1$, $\sum%
\limits_{k=1}^{n}(-1)^{k-1}x_{k}^{(j-i)}(t)v_{k,n}(t)=0$ by the inductive
hypothesis.

Thus $\hat{\alpha}^{(j)}(t)\cdot \hat{v}_{n+1}(t)=0$. To start the
induction, for $n=1$ we have $\hat{\alpha}(t)=\langle t\rangle \Rightarrow 
\hat{\alpha}^{(j)}(t)=0\Rightarrow \hat{\alpha}^{(j)}(t)\cdot \hat{v}%
_{1}(t)=0$ for $j\geq 2$.
\end{proof}

\begin{remark}
Proposition \ref{P1} could perhaps also be proven using properties of
hypergeometric functions.
\end{remark}

\section{Useful Determinants}

\begin{lemma}
\label{L5}For $n\geq 3$, $W(x_{1},...,x_{n})(t)=\dfrac{\tprod%
\limits_{r=0}^{n-1}r!}{t^{n(n-3)/2}},t>0$, where $W$ denotes the Wronskian.
\end{lemma}

\begin{proof}
It is easy to show that the $n\ $functions $\left\{ t{\large (}\log t{\large %
)}^{k-1}\right\} _{k=1,...,n}$ satisfy the following $n$th order Euler DE: 
\begin{equation*}
t^{n}\dfrac{d^{n}y}{dt^{n}}+\dfrac{n^{2}-3n}{2}t^{n-1}\dfrac{d^{n-1}y}{%
dt^{n-1}}+\cdots +a_{n-1}t\dfrac{dy}{dy}+a_{n}y=0.
\end{equation*}%
By Abel's Identity applied to the interval $(0,\infty )$, $%
W(x_{1},...,x_{n})(t)=$

$C_{n}\exp \left( -\dint \dfrac{\dfrac{n^{2}-3n}{2}t^{n-1}}{t^{n}}dt\right)
=C_{n}\exp \left( -\dfrac{n^{2}-3n}{2}\dint \dfrac{dt}{t}\right) =\dfrac{%
C_{n}}{t^{(n^{2}-3n)/2}}$. We shall let $t=1$ to obtain the precise value $%
C_{n}=\tprod\limits_{r=0}^{n-1}r!$. $W(x_{1},...,x_{n})(t)=\left\vert 
\begin{array}{llll}
x_{1}(t) & x_{2}(t) & \cdots & x_{n}(t) \\ 
x_{1}^{\prime }(t) & x_{2}^{\prime }(t) & \cdots & x_{n}^{\prime }(t) \\ 
x_{1}^{\prime \prime }(t) & x_{2}^{\prime \prime }(t) & \cdots & 
x_{n}^{\prime \prime }(t) \\ 
\vdots & \vdots & \cdots & \vdots \\ 
x_{1}^{(n-1)}(t) & x_{2}^{(n-1)}(t) & \cdots & x_{n}^{(n-1)}(t)%
\end{array}%
\right\vert =\left\vert 
\begin{array}{llll}
t & x_{2}(t) & \cdots & x_{n}(t) \\ 
1 & x_{2}^{\prime }(t) & \cdots & x_{n}^{\prime }(t) \\ 
0 & x_{2}^{\prime \prime }(t) & \cdots & x_{n}^{\prime \prime }(t) \\ 
\vdots & \vdots & \cdots & \vdots \\ 
0 & x_{2}^{(n-1)}(t) & \cdots & x_{n}^{(n-1)}(t)%
\end{array}%
\right\vert \Rightarrow W(x_{1},...,x_{n})(1)=\left\vert 
\begin{array}{llll}
1 & 0 & \cdots & 0 \\ 
1 & x_{2}^{\prime }(1) & \cdots & x_{n}^{\prime }(1) \\ 
0 & x_{2}^{\prime \prime }(1) & \cdots & x_{n}^{\prime \prime }(1) \\ 
\vdots & \vdots & \cdots & \vdots \\ 
0 & x_{2}^{(n-1)}(1) & \cdots & x_{n}^{(n-1)}(1)%
\end{array}%
\right\vert =$

$\left\vert 
\begin{array}{lll}
x_{2}^{\prime }(1) & \cdots & x_{n}^{\prime }(1) \\ 
x_{2}^{\prime \prime }(1) & \cdots & x_{n}^{\prime \prime }(1) \\ 
\vdots & \cdots & \vdots \\ 
x_{2}^{(n-1)}(1) & \cdots & x_{n}^{(n-1)}(1)%
\end{array}%
\right\vert $. The diagonal entries are $x_{r+1}^{(r)}(1),r=1,...,n-1$ and
for row $i$ we have $\left[ 
\begin{array}{lll}
x_{2}^{(i)}(1) & \cdots & x_{n}^{(i)}(1)%
\end{array}%
\right] $.

By Lemma \ref{L2}, the entries in row $i$, column $j,j\geq i+2$, are each $0$%
. That shows that the matrix $\left[ 
\begin{array}{lll}
x_{2}^{\prime }(1) & \cdots & x_{n}^{\prime }(1) \\ 
x_{2}^{\prime \prime }(1) & \cdots & x_{n}^{\prime \prime }(1) \\ 
\vdots & \cdots & \vdots \\ 
x_{2}^{(n-1)}(1) & \cdots & x_{n}^{(n-1)}(1)%
\end{array}%
\right] $ is upper triangular, which implies that $\left\vert 
\begin{array}{lll}
x_{2}^{\prime }(1) & \cdots & x_{n}^{\prime }(1) \\ 
x_{2}^{\prime \prime }(1) & \cdots & x_{n}^{\prime \prime }(1) \\ 
\vdots & \cdots & \vdots \\ 
x_{2}^{(n-1)}(1) & \cdots & x_{n}^{(n-1)}(1)%
\end{array}%
\right\vert
=\tprod\limits_{r=1}^{n-1}x_{r+1}^{(r)}(1)=\tprod\limits_{r=0}^{n-1}r!$ by
Lemma \ref{L2}.
\end{proof}

\begin{remark}
One could also prove Lemma \ref{L5}\textbf{\ }by instead finding a formula
for the Wronskian of $\left\{ {\large (}\log t{\large )}^{k-1}\right\}
_{k=1}^{n}$ and using well known properties of Wronskians. That would be
easier if one did not already have the recursion for $x_{k+1}^{(r)}(t)$.
Since we use that recursion elsewhere, it was easier to then prove Lemmas %
\ref{L2} first.
\end{remark}

Our next result shows that the Wronskians $W_{k,n}$ are in fact identically
equal to the functions $v_{k,n}$ from Lemma \ref{L6}.

\begin{proposition}
\textbf{\label{P2}}For $1\leq k\leq n$ and $n\geq $ $3$, 
\begin{equation*}
W_{k,n}(t)=\dfrac{\dprod\limits_{r=0}^{n-1}r!}{(k-1)!}\dfrac{1}{%
t^{(n-2)(n-1)/2}}\dsum\limits_{j=0}^{n-k}\dfrac{\left( \ln t\right) ^{n-k-j}%
}{(n-k-j)!},t>0.
\end{equation*}
\end{proposition}

\begin{proof}
Consider the following system of linear equations in the unknown functions $%
u_{1}(t),...,u_{n}(t)$, where $k_{n}(t)=\tfrac{\tprod\limits_{r=0}^{n-1}r!}{%
t^{n(n-3)/2}}$:%
\begin{eqnarray*}
x_{1}(t)u_{1}(t)+\cdots x_{n}(t)u_{n}(t) &=&k_{n}(t) \\
x_{1}^{\prime }(t)u_{1}(t)+\cdots x_{n}^{\prime }(t)u_{n}(t) &=&0 \\
&&\vdots \\
x_{1}^{(n-1)}(t)u_{1}(t)+\cdots x_{n}^{(n-1}(t)u_{n}(t) &=&0.
\end{eqnarray*}%
The coefficent matrix of this linear system has determinant $%
W(x_{1},...,x_{n})(t)$, which is nonzero by Lemma \ref{L5}. By Cramer's
Rule, the unique solution is given by 
\begin{eqnarray*}
x_{j}(t) &=& \\
&&\tfrac{\left\vert 
\begin{array}{lllllll}
x_{1}(t) & \cdots & x_{j-1}(t) & k_{n}(t) & x_{j+1}(t) & \cdots & x_{n}(t)
\\ 
x_{1}^{\prime }(t) & \cdots & x_{j-1}^{\prime }(t) & 0 & x_{j+1}^{\prime }(t)
& \cdots & x_{n}^{\prime }(t) \\ 
x_{1}^{\prime \prime }(t) & \cdots & x_{j-1}^{\prime \prime }(t) & 0 & 
x_{j+1}^{\prime \prime }(t) & \cdots & x_{n}^{\prime \prime }(t) \\ 
\vdots & \cdots & \vdots & \vdots & \vdots & \cdots & \vdots \\ 
x_{1}^{(n-1)}(t) & \cdots & x_{j-1}^{(n-1)}(t) & 0 & x_{j+1}^{(n-1)}(t) & 
\cdots & x_{n}^{(n-1)}(t)%
\end{array}%
\right\vert }{W(x_{1},...,x_{n})(t)} \\
j &=&1,...,n.
\end{eqnarray*}
Expand about column $j$ to obtain $x_{j}(t)=k_{n}(t)\tfrac{%
(-1)^{j+1}W_{j,n}(t)}{W(x_{1},...,x_{n})(t)}=(-1)^{j+1}W_{j,n}(t)$. By
Proposition 1, $x_{j}(t)=(-1)^{j+1}v_{j,n}(t)$ also satisfies (7). By
uniqueness, $W_{j,n}(t)=v_{j,n}(t),j=1,...,n$.
\end{proof}

\begin{remark}
It follows immediately from Proposition \ref{P2}, that the $W_{k,n}$ also
satisfy the following recursion from Lemma \ref{L6} for $n\geq 2$: 
\begin{eqnarray}
W_{k+1,n+1}(t) &=&\dfrac{n!}{k}\dfrac{1}{t^{n-1}}W_{k,n}(t),k\geq 1
\label{5} \\
W_{1,n+1}(t) &=&\left( \dprod\limits_{r=0}^{n-1}r!\right) \dfrac{\left( \ln
t\right) ^{n}}{t^{n(n-1)/2}}+\dfrac{n!}{t^{n-1}}W_{1,n}(t).  \label{6}
\end{eqnarray}%
(\ref{5}) can also be proven using the determinant definition of the $%
W_{k,n} $ along with standard properties of determinants. However, we found
it difficult to prove (\ref{6}) this way--hence the introduction of the $%
v_{k,n} $ functions.
\end{remark}

\begin{lemma}
\label{L7}For\textbf{\ }$n\geq 3$,%
\begin{equation*}
\dsum\limits_{k=1}^{n}\left[ (-1)^{k+1}b_{k}^{n-1}\tprod\limits_{1\leq
i<j\leq n\text{; }i,j\neq k}(b_{j}-b_{i})\right] =(-1)^{n-1}\tprod\limits_{1%
\leq i<j\leq n}(b_{j}-b_{i}).
\end{equation*}
\end{lemma}

\begin{proof}
It is well known that the Vandermonde determinant

$\left\vert 
\begin{array}{llll}
1 & 1 & \cdots & 1 \\ 
b_{1} & b_{2} & \cdots & b_{n} \\ 
\vdots & \vdots & \vdots & \vdots \\ 
b_{1}^{n-2} & b_{2}^{n-2} & \cdots & b_{n}^{n-2} \\ 
b_{1}^{n-1} & b_{2}^{n-1} & \cdots & b_{n}^{n-1}%
\end{array}%
\right\vert =\tprod\limits_{1\leq i<j\leq n}(b_{j}-b_{i})$, which implies
that

$\left\vert 
\begin{array}{llll}
b_{1}^{n-1} & b_{2}^{n-1} & \cdots & b_{n}^{n-1} \\ 
b_{1}^{n-2} & b_{2}^{n-2} & \cdots & b_{n}^{n-2} \\ 
\vdots & \vdots & \vdots & \vdots \\ 
b_{1} & b_{2} & \cdots & b_{n} \\ 
1 & 1 & \cdots & 1%
\end{array}%
\right\vert =(-1)^{n(n-1)/2}\tprod\limits_{1\leq i<j\leq n}(b_{j}-b_{i})$.
By expanding $\left\vert 
\begin{array}{llll}
b_{1}^{n-1} & b_{2}^{n-1} & \cdots & b_{n}^{n-1} \\ 
b_{1}^{n-2} & b_{2}^{n-2} & \cdots & b_{n}^{n-2} \\ 
\vdots & \vdots & \vdots & \vdots \\ 
b_{1} & b_{2} & \cdots & b_{n} \\ 
1 & 1 & \cdots & 1%
\end{array}%
\right\vert $ along the first row and using induction, one has $%
(-1)^{(n-1)(n-2)/2}\dsum\limits_{k=1}^{n}\left[ (-1)^{k+1}b_{k}^{n-1}\tprod%
\limits_{1\leq i<j\leq n\text{; }i,j\neq k}(b_{j}-b_{i})\right] =$

$(-1)^{n(n-1)/2}\tprod\limits_{1\leq i<j\leq n}(b_{j}-b_{i})$ and the lemma
follows immediately.
\end{proof}

\begin{proposition}
\textbf{\label{P3}}For $n\geq 3$, $%
\begin{vmatrix}
W_{1,n}(a_{1}) & -W_{2,n}(a_{1}) & \cdots & (-1)^{n+1}W_{n,n}(a_{1}) \\ 
W_{1,n}(a_{2}) & -W_{2,n}(a_{2}) & \cdots & (-1)^{n+1}W_{n,n}(a_{2}) \\ 
\vdots & \vdots & \vdots & \vdots \\ 
W_{1,n}(a_{n}) & -W_{2,n}(a_{n}) & \cdots & (-1)^{n+1}W_{n,n}(a_{n})%
\end{vmatrix}%
=\left( \dprod\limits_{r=0}^{n-1}r!\right) ^{n-2}\dfrac{\tprod\limits_{1\leq
i<j\leq n}(\ln a_{j}-\ln a_{i})}{\prod\limits_{j=1}^{n}a_{j}^{(n-1)(n-2)/2}}$%
.
\end{proposition}

\begin{proof}
We use induction. So assume that

$%
\begin{vmatrix}
{\small W}_{1,n-1}{\small (a}_{1}{\small )} & {\small -W}_{2,n-1}{\small (a}%
_{1}{\small )} & {\small \cdots } & {\small (-1)}^{n}{\small W}_{n-1,n-1}%
{\small (a}_{1}{\small )} \\ 
{\small W}_{1,n-1}{\small (a}_{2}{\small )} & {\small -W}_{2,n-1}{\small (a}%
_{2}{\small )} & {\small \cdots } & {\small (-1)}^{n}{\small W}_{n-1,n-1}%
{\small (a}_{2}{\small )} \\ 
{\small \vdots } & {\small \vdots } & {\small \vdots } & {\small \vdots } \\ 
{\small W}_{1,n-1}{\small (a}_{n-1}{\small )} & {\small -W}_{2,n-1}{\small (a%
}_{n-1}{\small )} & {\small \cdots } & {\small (-1)}^{n}{\small W}_{n-1,n-1}%
{\small (a}_{n-1}{\small )}%
\end{vmatrix}%
=$

$\left( \tprod\limits_{r=0}^{n-2}r!\right) ^{n-3}\tfrac{\tprod\limits_{1\leq
i<j\leq n-1}(\ln a_{j}-\ln a_{i})}{\prod%
\limits_{j=1}^{n-1}a_{j}^{(n-2)(n-3)/2}}$ for any $0<a_{1}<a_{2}<\cdots
<a_{n-1}$. Using (\ref{5}) we have $%
\begin{vmatrix}
{\small W}_{1,n}{\small (a}_{1}{\small )} & {\small -W}_{2,n}{\small (a}_{1}%
{\small )} & {\small \cdots } & {\small (-1)}^{n+1}{\small W}_{n,n}{\small (a%
}_{1}{\small )} \\ 
{\small W}_{1,n}{\small (a}_{2}{\small )} & {\small -W}_{2,n}{\small (a}_{2}%
{\small )} & {\small \cdots } & {\small (-1)}^{n+1}{\small W}_{n,n}{\small (a%
}_{2}{\small )} \\ 
{\small \vdots } & {\small \vdots } & {\small \vdots } & {\small \vdots } \\ 
{\small W}_{1,n}{\small (a}_{n}{\small )} & {\small -W}_{2,n}{\small (a}_{n}%
{\small )} & {\small \cdots } & {\small (-1)}^{n+1}{\small W}_{n,n}{\small (a%
}_{n}{\small )}%
\end{vmatrix}%
=$%
\begin{eqnarray*}
&&{\large (}(n-1)!{\large )}^{n-1}\times \\
&&%
\begin{vmatrix}
{\small W}_{1,n}{\small (a}_{1}{\small )} & {\small -}\tfrac{{\small W}%
_{1,n-1}{\small (a}_{1}{\small )}}{a_{1}^{n-2}} & {\small \cdots } & \tfrac{%
{\small (-1)}^{k}{\small W}_{k,n-1}{\small (a}_{1}{\small )}}{ka_{1}^{n-2}}
& {\small \cdots } & \tfrac{{\small (-1)}^{n+1}{\small W}_{n-1,n-1}{\small (a%
}_{1}{\small )}}{(n-1)a_{1}^{n-2}} \\ 
{\small W}_{1,n}{\small (a}_{2}{\small )} & {\small -}\tfrac{{\small W}%
_{1,n-1}{\small (a}_{2}{\small )}}{a_{2}^{n-2}} & {\small \cdots } & \tfrac{%
{\small (-1)}^{k}{\small W}_{k,n-1}{\small (a}_{2}{\small )}}{ka_{2}^{n-2}}
& {\small \cdots } & \tfrac{{\small (-1)}^{n+1}{\small W}_{n-1,n-1}{\small (a%
}_{2}{\small )}}{(n-1)a_{2}^{n-2}} \\ 
{\small \vdots } & {\small \vdots } & {\small \vdots } & {\small \vdots } & 
{\small \vdots } & {\small \vdots } \\ 
{\small W}_{1,n}{\small (a}_{n}{\small )} & {\small -}\tfrac{{\small W}%
_{1,n-1}{\small (a}_{n}{\small )}}{a_{n}^{n-2}} & {\small \cdots } & \tfrac{%
{\small (-1)}^{k}{\small W}_{k,n-1}{\small (a}_{n}{\small )}}{ka_{n}^{n-2}}
& {\small \cdots } & \tfrac{{\small (-1)}^{n+1}{\small W}_{n-1,n-1}{\small (a%
}_{n}{\small )}}{(n-1)a_{n}^{n-2}}%
\end{vmatrix}%
.
\end{eqnarray*}

Using (\ref{6}) yields 
\begin{eqnarray*}
&&{\large (}(n-1)!{\large )}^{n-1}\times \\
&&%
\begin{vmatrix}
\left( \tprod\limits_{r=0}^{n-2}r!\right) \tfrac{\left( \ln a_{1}\right)
^{n-1}}{a_{1}^{(n-1)(n-2)/2}}{\small +}\tfrac{(n-1)!{\small W}_{1,n-1}%
{\small (a}_{1}{\small )}}{a_{1}^{n-2}} & {\small -}\tfrac{{\small W}_{1,n-1}%
{\small (a}_{1}{\small )}}{a_{1}^{n-2}} & {\small \cdots } & \tfrac{{\small %
(-1)}^{n+1}{\small W}_{n-1,n-1}{\small (a}_{1}{\small )}}{(n-1)a_{1}^{n-2}}
\\ 
\left( \tprod\limits_{r=0}^{n-2}r!\right) \tfrac{\left( \ln a_{2}\right)
^{n-1}}{a_{2}^{(n-1)(n-2)/2}}{\small +}\tfrac{(n-1)!{\small W}_{1,n-1}%
{\small (a}_{2}{\small )}}{a_{2}^{n-2}} & {\small -}\tfrac{{\small W}_{1,n-1}%
{\small (a}_{2}{\small )}}{a_{2}^{n-2}} & {\small \cdots } & \tfrac{{\small %
(-1)}^{n+1}{\small W}_{n-1,n-1}{\small (a}_{2}{\small )}}{(n-1)a_{2}^{n-2}}
\\ 
{\small \vdots } & {\small \vdots } & {\small \vdots } & {\small \vdots } \\ 
\left( \tprod\limits_{r=0}^{n-2}r!\right) \tfrac{\left( \ln a_{n}\right)
^{n-1}}{a_{n}^{(n-1)(n-2)/2}}{\small +}\tfrac{(n-1)!{\small W}_{1,n-1}%
{\small (a}_{n}{\small )}}{a_{n}^{n-2}} & {\small -}\tfrac{{\small W}_{1,n-1}%
{\small (a}_{n}{\small )}}{a_{n}^{n-2}} & {\small \cdots } & \tfrac{{\small %
(-1)}^{n+1}{\small W}_{n-1,n-1}{\small (a}_{n}{\small )}}{(n-1)a_{n}^{n-2}}%
\end{vmatrix}%
.
\end{eqnarray*}

By adding $(n-1)!$ $\times $ Col. 2 to Col.1 we have 
\begin{eqnarray*}
&&{\large (}(n-1)!{\large )}^{n-1}\times \\
&&%
\begin{vmatrix}
\left( \tprod\limits_{r=0}^{n-2}r!\right) \tfrac{\left( \ln a_{1}\right)
^{n-1}}{a_{1}^{(n-1)(n-2)/2}} & {\small -}\tfrac{{\small W}_{1,n-1}{\small (a%
}_{1}{\small )}}{a_{1}^{n-2}} & {\small \cdots } & \tfrac{{\small (-1)}^{k}%
{\small W}_{k,n-1}{\small (a}_{1}{\small )}}{ka_{1}^{n-2}} & {\small \cdots }
& \tfrac{{\small (-1)}^{n+1}{\small W}_{n-1,n-1}{\small (a}_{1}{\small )}}{%
(n-1)a_{1}^{n-2}} \\ 
\left( \tprod\limits_{r=0}^{n-2}r!\right) \tfrac{\left( \ln a_{2}\right)
^{n-1}}{a_{2}^{(n-1)(n-2)/2}} & {\small -}\tfrac{{\small W}_{1,n-1}{\small (a%
}_{2}{\small )}}{a_{2}^{n-2}} & {\small \cdots } & \tfrac{{\small (-1)}^{k}%
{\small W}_{k,n-1}{\small (a}_{2}{\small )}}{ka_{2}^{n-2}} & {\small \cdots }
& \tfrac{{\small (-1)}^{n+1}{\small W}_{n-1,n-1}{\small (a}_{2}{\small )}}{%
(n-1)a_{2}^{n-2}} \\ 
\vdots & {\small \vdots } & {\small \vdots } & {\small \vdots } & {\small %
\vdots } & {\small \vdots } \\ 
\left( \tprod\limits_{r=0}^{n-2}r!\right) \tfrac{\left( \ln a_{n}\right)
^{n-1}}{a_{n}^{(n-1)(n-2)/2}} & {\small -}\tfrac{{\small W}_{1,n-1}{\small (a%
}_{n}{\small )}}{a_{n}^{n-2}} & {\small \cdots } & \tfrac{{\small (-1)}^{k}%
{\small W}_{k,n-1}{\small (a}_{n}{\small )}}{ka_{n}^{n-2}} & {\small \cdots }
& \tfrac{{\small (-1)}^{n+1}{\small W}_{n-1,n-1}{\small (a}_{n}{\small )}}{%
(n-1)a_{n}^{n-2}}%
\end{vmatrix}%
.
\end{eqnarray*}

Factoring out $\tprod\limits_{r=0}^{n-2}r!$ from Col. 1, factoring out $%
\tfrac{1}{k}$ from Column $k+1,k=1,..,n-1$, and factoring out $\tfrac{1}{%
a_{j}^{n-2}}$ from row $j,j=1,...,n$, yields 
\begin{eqnarray*}
&&\tfrac{{\large (}(n-1)!{\large )}^{n-2}\left(
\tprod\limits_{r=0}^{n-2}r!\right) }{\prod\limits_{j=1}^{n}a_{j}^{n-2}}\times
\\
&&%
\begin{vmatrix}
\tfrac{\left( \ln a_{1}\right) ^{n-1}}{a_{1}^{(n-2)(n-3)/2}} & {\small -W}%
_{1,n-1}{\small (a}_{1}{\small )} & {\small \cdots } & {\small (-1)}^{k}%
{\small W}_{k,n-1}{\small (a}_{1}{\small )} & {\small \cdots } & {\small (-1)%
}^{n+1}{\small W}_{n-1,n-1}{\small (a}_{1}{\small )} \\ 
\tfrac{\left( \ln a_{2}\right) ^{n-1}}{a_{2}^{(n-2)(n-3)/2}} & {\small -W}%
_{1,n-1}{\small (a}_{2}{\small )} & {\small \cdots } & {\small (-1)}^{k}%
{\small W}_{k,n-1}{\small (a}_{2}{\small )} & {\small \cdots } & {\small (-1)%
}^{n+1}{\small W}_{n-1,n-1}{\small (a}_{2}{\small )} \\ 
{\small \vdots } & {\small \vdots } & {\small \vdots } & {\small \vdots } & 
{\small \vdots } & {\small \vdots } \\ 
\tfrac{\left( \ln a_{n}\right) ^{n-1}}{a_{n}^{(n-2)(n-3)/2}} & {\small -W}%
_{1,n-1}{\small (a}_{n}{\small )} & {\small \cdots } & {\small (-1)}^{k}%
{\small W}_{k,n-1}{\small (a}_{n}{\small )} & {\small \cdots } & {\small (-1)%
}^{n+1}{\small W}_{n-1,n-1}{\small (a}_{n}{\small )}%
\end{vmatrix}%
.
\end{eqnarray*}%
By expanding about Col. 1 we obtain 
\begin{eqnarray*}
&&\tfrac{{\large (}(n-1)!{\large )}^{n-2}\left(
\tprod\limits_{r=0}^{n-2}r!\right) }{\prod\limits_{j=1}^{n}a_{j}^{n-2}}\times
\\
&&\left( 
\begin{array}{c}
\tfrac{\left( \ln a_{1}\right) ^{n-1}}{a_{1}^{(n-2)(n-3)/2}}%
\begin{vmatrix}
{\small -W}_{1,n-1}{\small (a}_{2}{\small )} & {\small \cdots } & {\small %
(-1)}^{n+1}{\small W}_{n-1,n-1}{\small (a}_{2}{\small )} \\ 
{\small \vdots } & {\small \vdots } & {\small \vdots } \\ 
{\small -W}_{1,n-1}{\small (a}_{n}{\small )} & {\small \cdots } & {\small %
(-1)}^{n+1}{\small W}_{n-1,n-1}{\small (a}_{n}{\small )}%
\end{vmatrix}%
- \\ 
\tfrac{\left( \ln a_{2}\right) ^{n-1}}{a_{2}^{(n-2)(n-3)/2}}%
\begin{vmatrix}
{\small -W}_{1,n-1}{\small (a}_{1}{\small )} & {\small \cdots } & {\small %
(-1)}^{n+1}{\small W}_{n-1,n-1}{\small (a}_{1}{\small )} \\ 
{\small -W}_{1,n-1}{\small (a}_{3}{\small )} & {\small \cdots } & {\small %
(-1)}^{n+1}{\small W}_{n-1,n-1}{\small (a}_{3}{\small )} \\ 
{\small \vdots } & {\small \vdots } & {\small \vdots } \\ 
{\small -W}_{1,n-1}{\small (a}_{n}{\small )} & {\small \cdots } & {\small %
(-1)}^{n+1}{\small W}_{n-1,n-1}{\small (a}_{n}{\small )}%
\end{vmatrix}%
+\cdots + \\ 
(-1)^{n+1}\tfrac{\left( \ln a_{n}\right) ^{n-1}}{a_{n}^{(n-2)(n-3)/2}}%
\begin{vmatrix}
{\small -W}_{1,n-1}{\small (a}_{1}{\small )} & {\small \cdots } & {\small %
(-1)}^{n+1}{\small W}_{n-1,n-1}{\small (a}_{1}{\small )} \\ 
{\small \vdots } & {\small \vdots } & {\small \vdots } \\ 
{\small -W}_{1,n-1}{\small (a}_{n-1}{\small )} & {\small \cdots } & {\small %
(-1)}^{n+1}{\small W}_{n-1,n-1}{\small (a}_{n-1}{\small )}%
\end{vmatrix}%
\end{array}%
\right) .
\end{eqnarray*}%
Factoring out $-1$ from each column of each determinant and using the
induction hypothesis gives 
\begin{eqnarray*}
&&(-1)^{n-1}\tfrac{{\large (}(n-1)!{\large )}^{n-2}\left(
\tprod\limits_{r=0}^{n-2}r!\right) }{\prod\limits_{j=1}^{n}a_{j}^{n-2}}\times
\\
&&\left( 
\begin{array}{c}
\left( \tprod\limits_{r=0}^{n-2}r!\right) ^{n-3}\tfrac{\left( \ln
a_{1}\right) ^{n-1}}{a_{1}^{(n-2)(n-3)/2}}\tfrac{\tprod\limits_{2\leq
i<j\leq n}(\ln a_{j}-\ln a_{i})}{\prod\limits_{j=2}^{n}a_{j}^{(n-2)(n-3)/2}}%
+\cdots + \\ 
(-1)^{n+1}\left( \tprod\limits_{r=0}^{n-2}r!\right) ^{n-3}\tfrac{\left( \ln
a_{n}\right) ^{n-1}}{a_{n}^{(n-2)(n-3)/2}}\tfrac{\tprod\limits_{1\leq
i<j\leq n-1}(\ln a_{j}-\ln a_{i})}{\prod%
\limits_{j=1}^{n-1}a_{j}^{(n-2)(n-3)/2}}%
\end{array}%
\right)
\end{eqnarray*}%
\begin{eqnarray*}
&=&(-1)^{n-1}\tfrac{\left( \tprod\limits_{r=0}^{n-1}r!\right) ^{n-2}}{%
\prod\limits_{j=1}^{n}a_{j}^{(n-1)(n-2)/2}}\times \\
&&{\large (}\left( \ln a_{1}\right) ^{n-1}\tprod\limits_{2\leq i<j\leq
n}(\ln a_{j}-\ln a_{i})+\cdots \\
&&+(-1)^{n+1}\left( \ln a_{n}\right) ^{n-1}\tprod\limits_{1\leq i<j\leq
n-1}(\ln a_{j}-\ln a_{i}){\large )}.
\end{eqnarray*}%
Applying Lemma \ref{L7} to each term of the sum in parentheses yields

$\tfrac{\left( \tprod\limits_{r=0}^{n-1}r!\right) ^{n-2}}{%
\prod\limits_{j=1}^{n}a_{j}^{(n-1)(n-2)/2}}\tprod\limits_{1\leq i<j\leq
n}(\ln a_{j}-\ln a_{i})$. For $n=3$ we have $W_{1,3}(t)=\tfrac{\ln
^{2}t+2\ln t+2}{t}$, $W_{2,3}(t)=2\tfrac{\ln t+1}{t}$, and $W_{3,3}(t)=%
\tfrac{1}{t}$. Thus $\left\vert 
\begin{array}{lll}
{\small W}_{1,3}{\small (a}_{1}{\small )} & {\small -W}_{2,3}{\small (a}_{1}%
{\small )} & {\small W}_{3,3}{\small (a}_{1}{\small )} \\ 
{\small W}_{1,3}{\small (a}_{2}{\small )} & {\small -W}_{2,3}{\small (a}_{2}%
{\small )} & {\small W}_{3,3}{\small (a}_{2}{\small )} \\ 
{\small W}_{1,3}{\small (a}_{3}{\small )} & {\small W}_{2,3}{\small (a}_{3}%
{\small )} & {\small W}_{3,3}{\small (a}_{3}{\small )}%
\end{array}%
\right\vert =$

$\left\vert 
\begin{array}{lll}
\tfrac{\ln ^{2}a_{1}+2\ln a_{1}+2}{a_{1}} & -{\small 2}\tfrac{\ln a_{1}+1}{%
a_{1}} & \tfrac{1}{a_{1}} \\ 
\tfrac{\ln ^{2}a_{2}+2\ln a_{2}+2}{a_{2}} & -{\small 2}\tfrac{\ln a_{2}+1}{%
a_{2}} & \tfrac{1}{a_{2}} \\ 
\tfrac{\ln ^{2}a_{3}+2\ln a_{3}+2}{a_{3}} & -{\small 2}\tfrac{\ln a_{3}+1}{%
a_{3}} & \tfrac{1}{a_{3}}%
\end{array}%
\right\vert =\allowbreak \tfrac{2\left( \ln a_{3}-\ln a_{1}\right) \left(
\ln a_{3}-\ln a_{2}\right) (\ln a_{2}-\ln a_{1})}{a_{1}a_{2}a_{3}}%
\allowbreak $(after some simplification), which equals

$\left( \tprod\limits_{r=0}^{n-1}r!\right) ^{n-2}\tfrac{\tprod\limits_{1\leq
i<j\leq n}(\ln a_{j}-\ln a_{i})}{\prod\limits_{j=1}^{n}a_{j}^{(n-1)(n-2)/2}}$
for $n=3$.
\end{proof}

\begin{proposition}
\textbf{\label{P4}}For $n\geq 3$, let $k_{n}(t)=\dfrac{\dprod%
\limits_{r=0}^{n-1}r!}{t^{n(n-3)/2}}$. Then 
\begin{gather*}
\begin{vmatrix}
k_{n}(a_{1}) & -W_{2,n}(a_{1}) & \cdots & (-1)^{n+1}W_{n,n}(a_{1}) \\ 
k_{n}(a_{2}) & -W_{2,n}(a_{2}) & \cdots & (-1)^{n+1}W_{n,n}(a_{2}) \\ 
\vdots & \vdots & \vdots & \vdots \\ 
k_{n}(a_{n}) & -W_{2,n}(a_{n}) & \cdots & (-1)^{n+1}W_{n,n}(a_{n})%
\end{vmatrix}%
= \\
(-1)^{n-1}(n-1)!\left( \dprod\limits_{r=0}^{n-1}r!\right) ^{n-2}\dfrac{%
\dsum\limits_{i=1}^{n}\left( \prod\limits_{\substack{ 1\leq j<k\leq n  \\ %
j\neq i,k\neq i}}(-1)^{i+1}a_{i}(\ln a_{k}-\ln a_{j})\right) }{%
\prod\limits_{j=1}^{n}a_{j}^{(n-1)(n-2)/2}}.
\end{gather*}
\end{proposition}

\begin{proof}
We again use induction. So assume that

$%
\begin{vmatrix}
{\small k}_{n-1}{\small (a_{1})} & {\small -W}_{2,n-1}{\small (a_{1})} & 
{\small \cdots } & {\small (-1)}^{n}{\small W}_{n-1,n-1}{\small (a_{1})} \\ 
{\small k}_{n-1}{\small (a_{2})} & {\small -W}_{2,n-1}{\small (a_{2})} & 
{\small \cdots } & {\small (-1)}^{n}{\small W}_{n-1,n-1}{\small (a_{2})} \\ 
{\small \vdots } & {\small \vdots } & {\small \vdots } & {\small \vdots } \\ 
{\small k}_{n-1}{\small (a}_{n-1}{\small )} & {\small -W}_{2,n-1}{\small %
(a_{n-1})} & {\small \cdots } & {\small (-1)}^{n}{\small W}_{n-1,n-1}{\small %
(a_{n-1})}%
\end{vmatrix}%
=$

$(-1)^{n}(n-2)!\left( \tprod\limits_{r=0}^{n-2}r!\right) ^{n-3}\tfrac{%
\dsum\limits_{i=1}^{n-1}\left( \prod\limits_{\substack{ 1\leq j<k\leq n-1 
\\ j\neq i,k\neq i}}(-1)^{i+1}a_{i}(\ln a_{k}-\ln a_{j})\right) }{%
\prod\limits_{m=1}^{n-1}a_{j}^{(n-2)(n-3)/2}}$ for any

$0<a_{1}<a_{2}<\cdots <a_{n-1}$. Using (\ref{5}) we have 
\begin{gather*}
\begin{vmatrix}
{\small k}_{n}{\small (a}_{1}{\small )} & {\small -W}_{2,n}{\small (a}_{1}%
{\small )} & {\small \cdots } & {\small (-1)}^{n+1}{\small W}_{n,n}{\small (a%
}_{1}{\small )} \\ 
{\small k}_{n}{\small (a}_{2}{\small )} & {\small -W}_{2,n}{\small (a}_{2}%
{\small )} & {\small \cdots } & {\small (-1)}^{n+1}{\small W}_{n,n}{\small (a%
}_{2}{\small )} \\ 
{\small \vdots } & {\small \vdots } & {\small \vdots } & {\small \vdots } \\ 
{\small k}_{n}{\small (a}_{n}{\small )} & {\small -W}_{2,n}{\small (a}_{n}%
{\small )} & {\small \cdots } & {\small (-1)}^{n+1}{\small W}_{n,n}{\small (a%
}_{n}{\small )}%
\end{vmatrix}%
= \\
\begin{vmatrix}
\tfrac{\tprod\limits_{r=0}^{n-1}r!}{a_{1}^{n(n-3)/2}} & -\tfrac{%
(n-1)!W_{1,n-1}(a_{1}}{a_{1}^{n-2}}) & \cdots & \tfrac{%
(-1)^{n+1}(n-1)!W_{n-1,n-1}(a_{1})}{a_{1}^{n-2}(n-1)} \\ 
\tfrac{\tprod\limits_{r=0}^{n-1}r!}{a_{2}^{n(n-3)/2}} & -\tfrac{%
(n-1)!W_{1,n-1}(a_{2})}{a_{2}^{n-2}} & \cdots & \tfrac{%
(-1)^{n+1}(n-1)!W_{n-1,n-1}(a_{2})}{a_{2}^{n-2}(n-1)} \\ 
\vdots & \vdots & \vdots & \vdots \\ 
\tfrac{\tprod\limits_{r=0}^{n-1}r!}{a_{n}^{n(n-3)/2}} & -\tfrac{%
(n-1)!W_{1,n-1}(a_{n})}{a_{n}^{n-2}} & \cdots & \tfrac{%
(-1)^{n+1}(n-1)!W_{n-1,n-1}(a_{n})}{a_{n}^{n-2}(n-1)}%
\end{vmatrix}%
.
\end{gather*}

Factoring out $\tprod\limits_{r=0}^{n-1}r!$ from Col. 1, factoring out $%
\tfrac{(n-1)!}{k}$ from Column $k+1,k=1,..,n-1$, and factoring out $\tfrac{1%
}{a_{j}^{n-2}}$ from row $j,j=1,...,n$, we obtain

$\tfrac{{\large (}(n-1)!{\large )}^{n-2}\tprod\limits_{r=0}^{n-1}r!}{%
\prod\limits_{j=1}^{n}a_{j}^{n-2}}%
\begin{vmatrix}
\tfrac{1}{a_{1}^{(n-1)(n-4)/2}} & {\small -W}_{1,n-1}{\small (a}_{1}{\small )%
} & {\small \cdots } & {\small (-1)}^{n+1}{\small W}_{n-1,n-1}{\small (a}_{1}%
{\small )} \\ 
\tfrac{1}{a_{2}^{(n-1)(n-4)/2}} & {\small -W}_{1,n-1}{\small (a}_{2}{\small )%
} & {\small \cdots } & {\small (-1)}^{n+1}{\small W}_{n-1,n-1}{\small (a}_{2}%
{\small )} \\ 
{\small \vdots } & {\small \vdots } & {\small \vdots } & {\small \vdots } \\ 
\tfrac{1}{a_{n}^{(n-1)(n-4)/2}} & {\small -W}_{1,n-1}{\small (a}_{n}{\small )%
} & {\small \cdots } & {\small (-1)}^{n+1}{\small W}_{n-1,n-1}{\small (a}_{n}%
{\small )}%
\end{vmatrix}%
$.

Expanding about Col. 1 gives

\begin{eqnarray*}
&&\tfrac{{\large (}(n-1)!{\large )}^{n-2}\tprod\limits_{r=0}^{n-1}r!}{%
\prod\limits_{j=1}^{n}a_{j}^{n-2}}\times \\
&&\left( 
\begin{array}{c}
\tfrac{1}{a_{1}^{(n-1)(n-4)/2}}%
\begin{vmatrix}
{\small -W}_{1,n-1}{\small (a}_{2}{\small )} & {\small \cdots } & {\small %
(-1)}^{n+1}{\small W}_{n-1,n-1}{\small (a}_{2}{\small )} \\ 
{\small \vdots } & {\small \vdots } & {\small \vdots } \\ 
{\small -W}_{1,n-1}{\small (a}_{n}{\small )} & {\small \cdots } & {\small %
(-1)}^{n+1}{\small W}_{n-1,n-1}{\small (a}_{n}{\small )}%
\end{vmatrix}%
- \\ 
\tfrac{1}{a_{2}^{(n-1)(n-4)/2}}%
\begin{vmatrix}
{\small -W}_{1,n-1}{\small (a}_{1}{\small )} & {\small \cdots } & {\small %
(-1)}^{n+1}{\small W}_{n-1,n-1}{\small (a}_{1}{\small )} \\ 
{\small -W}_{1,n-1}{\small (a}_{3}{\small )} & {\small \cdots } & {\small %
(-1)}^{n+1}{\small W}_{n-1,n-1}{\small (a}_{3}{\small )} \\ 
{\small \vdots } & {\small \vdots } & {\small \vdots } \\ 
{\small -W}_{1,n-1}{\small (a}_{n}{\small )} & {\small \cdots } & {\small %
(-1)}^{n+1}{\small W}_{n-1,n-1}{\small (a}_{n}{\small )}%
\end{vmatrix}%
+\cdots \\ 
+(-1)^{n+1}\tfrac{1}{a_{n}^{(n-1)(n-4)/2}}%
\begin{vmatrix}
{\small -W}_{1,n-1}{\small (a}_{1}{\small )} & {\small \cdots } & {\small %
(-1)}^{n+1}{\small W}_{n-1,n-1}{\small (a}_{1}{\small )} \\ 
{\small \vdots } & {\small \vdots } & {\small \vdots } \\ 
{\small -W}_{1,n-1}{\small (a}_{n-1}{\small )} & {\small \cdots } & {\small %
(-1)}^{n+1}{\small W}_{n-1,n-1}{\small (a}_{n-1}{\small )}%
\end{vmatrix}%
.%
\end{array}%
\right)
\end{eqnarray*}%
.

Factoring out $-1$ from each column of each determinant and using
Proposition \ref{P3} yields 
\begin{eqnarray*}
&&(-1)^{n-1}\tfrac{{\large (}(n-1)!{\large )}^{n-2}\tprod%
\limits_{r=0}^{n-1}r!}{\prod\limits_{j=1}^{n}a_{j}^{n-2}}\times \\
&&\left( 
\begin{array}{c}
\tfrac{1}{a_{1}^{(n-1)(n-4)/2}}\left( \tprod\limits_{r=0}^{n-2}r!\right)
^{n-3}\tfrac{\tprod\limits_{2\leq i<j\leq n}(\ln a_{j}-\ln a_{i})}{%
\prod\limits_{j=2}^{n}a_{j}^{(n-2)(n-3)/2}}+\cdots + \\ 
(-1)^{n+1}\tfrac{1}{a_{n}^{(n-1)(n-4)/2}}\left(
\tprod\limits_{r=0}^{n-2}r!\right) ^{n-3}\tfrac{\tprod\limits_{1\leq i<j\leq
n-1}(\ln a_{j}-\ln a_{i})}{\prod\limits_{j=1}^{n-1}a_{j}^{(n-2)(n-3)/2}}%
\end{array}%
\right)
\end{eqnarray*}%
$=$

\begin{gather*}
=(-1)^{n-1}\tfrac{{\large (}(n-1)!{\large )}^{n-2}\left(
\tprod\limits_{r=0}^{n-1}r!\right) \left( \tprod\limits_{r=0}^{n-2}r!\right)
^{n-3}}{\prod\limits_{j=1}^{n}a_{j}^{n-2}}\times \\
\left( \tfrac{a_{1}\tprod\limits_{2\leq i<j\leq n}(\ln a_{j}-\ln a_{i})}{%
\prod\limits_{j=1}^{n}a_{j}^{(n-2)(n-3)/2}}+\cdots +(-1)^{n+1}\tfrac{%
a_{n}\tprod\limits_{1\leq i<j\leq n-1}(\ln a_{j}-\ln a_{i})}{%
\prod\limits_{j=1}^{n}a_{j}^{(n-2)(n-3)/2}}\right)
\end{gather*}

$=(-1)^{n-1}(n-1)!\left( \tprod\limits_{r=0}^{n-1}r!\right) ^{n-2}\tfrac{%
\dsum\limits_{i=1}^{n}\left( \prod\limits_{\substack{ 1\leq j<k\leq n  \\ %
j\neq i,k\neq i}}(-1)^{i+1}a_{i}(\ln a_{k}-\ln a_{j})\right) }{%
\prod\limits_{j=1}^{n}a_{j}^{(n-1)(n-2)/2}}$. For $n=3$ we have $\left\vert 
\begin{array}{lll}
{\small k}_{3}{\small (a}_{1}{\small )} & {\small -W}_{2,3}{\small (a}_{1}%
{\small )} & {\small W}_{3,3}{\small (a}_{1}{\small )} \\ 
{\small k}_{3}{\small (a}_{2}{\small )} & {\small -W}_{2,3}{\small (a}_{2}%
{\small )} & {\small W}_{3,3}{\small (a}_{2}{\small )} \\ 
{\small k}_{3}{\small (a}_{3}{\small )} & {\small W}_{2,3}{\small (a}_{3}%
{\small )} & {\small W}_{3,3}{\small (a}_{3}{\small )}%
\end{array}%
\right\vert =$

$\left\vert 
\begin{array}{lll}
2 & -{\small 2}\tfrac{\ln a_{1}+1}{a_{1}} & \tfrac{1}{a_{1}} \\ 
2 & -{\small 2}\tfrac{\ln a_{2}+1}{a_{2}} & \tfrac{1}{a_{2}} \\ 
2 & -{\small 2}\tfrac{\ln a_{3}+1}{a_{3}} & \tfrac{1}{a_{3}}%
\end{array}%
\right\vert =\allowbreak 4\tfrac{a_{1}\left( \ln a_{3}-\ln a_{2}\right)
-a_{2}\left( \ln a_{3}-\ln a_{1}\right) +a_{3}(\ln a_{2}-\ln a_{1})}{%
a_{1}a_{2}a_{3}}\allowbreak $ $\allowbreak $(after some simplification),
which equals

$(-1)^{n-1}(n-1)!\left( \tprod\limits_{r=0}^{n-1}r!\right) ^{n-2}\tfrac{%
\dsum\limits_{i=1}^{n}\left( \prod\limits_{\substack{ 1\leq j<k\leq n  \\ %
j\neq i,k\neq i}}(-1)^{i+1}a_{i}(\ln a_{k}-\ln a_{j})\right) }{%
\prod\limits_{j=1}^{n}a_{j}^{(n-1)(n-2)/2}}$ for $n=3$\textbf{.}
\end{proof}

\section{Proof of Theorem \protect\ref{main}}

\begin{proof}
The equation of the osculating hyperplane, $O_{a}$, to $C$ at $t=a$\ is $%
\langle x_{1},...,x_{n}\rangle \cdot \hat{n}(a)=\hat{\alpha}(a)\cdot \hat{n}%
(a)$, where\ $\hat{\alpha}(t)=\langle t,t\log t,...,t{\large (}\log t{\large %
)}^{n-1}\rangle ,$

$\hat{n}(t)=\langle W_{1,n}(t),-W_{2,n}(t),...,(-1)^{n+1}W_{n,n}(t)\rangle
,W_{j,n}(t)=$ the Wronskian of $x_{1}^{\prime }(t),...,x_{j-1}^{\prime
}(t),x_{j+1}^{\prime }(t),...,x_{n}^{\prime }(t),j=1,...,n$. Thus any
intersection point of $O_{1},...,O_{n}$ must be a solution of the linear
system $\langle x_{1},...,x_{n}\rangle \cdot \hat{n}(a_{j})=\hat{\alpha}%
(a_{j})\cdot \hat{n}(a_{j})=0,j=1,...,n$, which can be written in the form 
\begin{gather}
\left[ 
\begin{array}{cccc}
W_{1,n}(a_{1}) & -W_{2,n}(a_{1}) & \cdots & (-1)^{n+1}W_{n,n}(a_{1}) \\ 
W_{1,n}(a_{2}) & -W_{2,n}(a_{2}) & \cdots & (-1)^{n+1}W_{n,n}(a_{2}) \\ 
\vdots & \vdots & \vdots & \vdots \\ 
W_{1,n}(a_{n}) & -W_{2,n}(a_{n}) & \cdots & (-1)^{n+1}W_{n,n}(a_{n})%
\end{array}%
\right] \left[ 
\begin{array}{l}
x_{1} \\ 
\vdots \\ 
x_{n}%
\end{array}%
\right] =  \label{7} \\
\left[ 
\begin{array}{l}
k(a_{1}) \\ 
\vdots \\ 
k(a_{n})%
\end{array}%
\right] ,  \notag
\end{gather}%
where $k(t)=\hat{\alpha}(t)\cdot \hat{n}(t)=W(x_{1},...,x_{n})(t)=\dfrac{%
\dprod\limits_{r=0}^{n-1}r!}{t^{n(n-3)/2}}$ by Lemma \ref{L5}. (\ref{7}) has
a unique solution, $x_{1},...,x_{n}$, by Proposition \ref{P3}. By Cramer's
Rule, 
\begin{eqnarray*}
x_{1} &=& \\
&&%
\begin{vmatrix}
{\small W}_{1,n}{\small (a}_{1}{\small )} & {\small -W}_{2,n}{\small (a}_{1}%
{\small )} & {\small \cdots } & {\small (-1)}^{n+1}{\small W}_{n,n}{\small (a%
}_{1}{\small )} \\ 
{\small W}_{1,n}{\small (a}_{2}{\small )} & {\small -W}_{2,n}{\small (a}_{2}%
{\small )} & {\small \cdots } & {\small (-1)}^{n+1}{\small W}_{n,n}{\small (a%
}_{2}{\small )} \\ 
{\small \vdots } & {\small \vdots } & {\small \vdots } & {\small \vdots } \\ 
{\small W}_{1,n}{\small (a}_{n}{\small )} & {\small -W}_{2,n}{\small (a}_{n}%
{\small )} & {\small \cdots } & {\small (-1)}^{n+1}{\small W}_{n,n}{\small (a%
}_{n}{\small )}%
\end{vmatrix}%
/ \\
&&%
\begin{vmatrix}
{\small k}_{n}{\small (a}_{1}{\small )} & {\small -W}_{2}{\small (a}_{1}%
{\small )} & {\small \cdots } & {\small (-1)}^{n+1}{\small W}_{n}{\small (a}%
_{1}{\small )} \\ 
{\small k}_{n}{\small (a}_{2}{\small )} & {\small -W}_{2}{\small (a}_{2}%
{\small )} & {\small \cdots } & {\small (-1)}^{n+1}{\small W}_{n}{\small (a}%
_{2}{\small )} \\ 
{\small \vdots } & {\small \vdots } & {\small \vdots } & {\small \vdots } \\ 
{\small k}_{n}{\small (a}_{n}{\small )} & {\small -W}_{2}{\small (a}_{n}%
{\small )} & {\small \cdots } & {\small (-1)}^{n+1}{\small W}_{n}{\small (a}%
_{n}{\small )}%
\end{vmatrix}%
\end{eqnarray*}%
$=\tfrac{(-1)^{n-1}(n-1)!\left( \tprod\limits_{r=0}^{n-1}r!\right) ^{n-2}%
\tfrac{\dsum\limits_{i=1}^{n}\left( \prod\limits_{\substack{ 1\leq j<k\leq n 
\\ j\neq i,k\neq i}}(-1)^{i+1}a_{i}(\ln a_{k}-\ln a_{j})\right) }{%
\prod\limits_{j=1}^{n}a_{j}^{(n-1)(n-2)/2}}}{\left(
\tprod\limits_{r=0}^{n-1}r!\right) ^{n-2}\tfrac{\tprod\limits_{1\leq i<j\leq
n}(\ln a_{j}-\ln a_{i})}{\prod\limits_{j=1}^{n}a_{j}^{(n-1)(n-2)/2}}}$ by
Propositions \ref{P3} and \ref{P4}. Simplifying gives $\tfrac{%
(n-1)!\dsum\limits_{i=1}^{n}\left( \prod\limits_{\substack{ 1\leq j<k\leq n 
\\ j\neq i,k\neq i}}(-1)^{n+i}a_{i}(\ln a_{k}-\ln a_{j})\right) }{%
\tprod\limits_{1\leq i<j\leq n}(\ln a_{j}-\ln a_{i})}$. By getting a common
denominator in the right hand side of (\ref{1}), it then follows easily that
the latter expression equals $(n-1)!\dsum\limits_{j=1}^{n}\tfrac{a_{j}}{%
\prod\limits_{\substack{ i=1  \\ i\neq j}}^{n}(\ln a_{j}-\ln a_{i})}$.
\end{proof}

\begin{remark}
In \cite{XZ} the following extension of the Identric mean $I(a,b)=\left( 
\dfrac{a^{a}}{b^{b}}\right) ^{1/(a-b)}/e$ to $n$ variables was given:

$I_{Z}(a_{1},...,a_{n})=\exp \left[ \dfrac{1}{V(a)}\dsum%
\limits_{i=1}^{n}(-1)^{n+i}a_{i}^{n-1}V_{i}(a)\ln a_{i}-m\right] $, where $%
V(a_{1},...,a_{n})=\tprod\limits_{1\leq j<i\leq
n}^{n}(a_{i}-a_{j}),V_{i}(a_{1},...,a_{n})=\left\vert 
\begin{array}{ccccccc}
{\small 1} & {\small 1} & {\small ...} & {\small 1} & {\small 1} & {\small %
...} & {\small 1} \\ 
{\small a}_{1} & {\small a}_{2} & {\small ...} & {\small a}_{i-1} & {\small a%
}_{i+1} & {\small ...} & {\small a}_{n} \\ 
{\small a}_{1}^{2} & {\small a}_{2}^{2} & {\small ...} & {\small a}_{i-1}^{2}
& {\small a}_{i+1}^{2} & {\small ...} & {\small a}_{n}^{2} \\ 
{\small ...} & {\small ...} & {\small ...} & {\small ...} & {\small ...} & 
{\small ...} & {\small ...} \\ 
{\small a}_{1}^{n-2} & {\small a}_{2}^{n-2} & {\small ...} & {\small a}%
_{i-1}^{n-2} & {\small a}_{i+1}^{n-2} & {\small ...} & {\small a}_{n}^{n-2}%
\end{array}%
\right\vert $, and $m=\dsum\limits_{k=1}^{n-1}\dfrac{1}{k}$. For $n=3$, if
one lets $x_{1}(t)=t$, $x_{2}(t)=t^{2}$, $x_{3}(t)=\log t$, then $%
M_{z}(a,b,c)=I_{Z}(a,b,c)=U_{2}(a,b,c)$, where $U_{2}$ is given in \cite{STO}%
. This probably holds for all $n$.
\end{remark}

\begin{conjecture}
If $x_{1}(t)=t$, $x_{2}(t)=t^{2}$, $...,x_{n-1}(t)=t^{n-1},x_{n}(t)=\log t$,
then $M_{n}\left( a_{1},...,a_{n}\right) =I_{Z}(a_{1},...,a_{n})$.
\end{conjecture}

This conjecture is probably somewhat easier to prove than Theorem \ref{main}.


\begin{thebibliography}{9}
\bibitem{H} A. Horwitz, \textquotedblleft Means, Generalized Divided
Differences, and Intersections of Osculating Hyperplanes\textquotedblright ,
JMAA 200(1996), 126--148.

\bibitem{M} Jorma K. Merikoski, "Extending Means Of Two Variables To Several
Variables, JIPAM, Volume 5, Issue 3, Article 65, 2004.

\bibitem{MUS} S. Mustonen, \textquotedblleft Logarithmic mean for several
arguments\textquotedblright , (2002). ONLINE
[http://www.survo.fi/papers/logmean.pdf].

\bibitem{N} Edward Neuman, \textquotedblleft The Weighted Logarithmic
Mean\textquotedblright , JMAA 188, 885--900(1994).

\bibitem{P} A. O. Pittenger, The logarithmic mean in n variables, Amer.
Math. Monthly 92(1985), 99--104.

\bibitem{STO} K. Stolarsky, \textquotedblleft Generalizations of the
logarithmic mean\textquotedblright , Math. Mag. 48 (1975), 87--92.

\bibitem{XZ} Zhen-Gang Xiao and Zhi-Hua Zhang, \textquotedblleft The
Inequalities $G\leq L\leq I\leq A$ in $n$ Variables\textquotedblright ,
JIPAM, Volume 4, Issue 2, Article 39, 2003.
\end{thebibliography}
\end{document}